\documentclass[12pt,reqno]{amsart}
\usepackage[margin=2.0cm]{geometry}

\usepackage{amsmath,amssymb,latexsym,amsthm,amstext}
\usepackage{bbm}
\usepackage[nobysame,alphabetic,initials]{amsrefs}
\usepackage{enumerate}

\usepackage{graphicx}
\usepackage{todonotes}

\usepackage{pgfplots}
\pgfplotsset{compat=1.15}
\usepackage{mathrsfs}
\usetikzlibrary{arrows}
\usetikzlibrary{intersections,angles,quotes}

\usepackage{float}

\usepackage{nicefrac}

\usepackage{hyperref}

\usepackage{orcidlink}
\usepackage{marginnote} 
\usepackage{soul} 
\usepackage{cancel} 

\usepackage{mathtools}
\mathtoolsset{showonlyrefs}

\newtheorem{theorem}{Theorem}[section]

\newtheorem{lemma}[theorem]{Lemma}
\newtheorem{prop}[theorem]{Proposition}
\newtheorem{corollary}[theorem]{Corollary}

\theoremstyle{remark}
\newtheorem{case}{Case}

\newcommand{\R}{{\mathbb R}}

\newcommand{\hyp}{{\mathbb H}}
\newcommand{\sph}{{\mathbb S}}
\newcommand{\M}{\mathcal{M}}

\DeclareMathOperator{\arcosh}{\mathrm{arcosh}}

\DeclareMathOperator{\arccot}{\mathrm{arccot}}
\DeclareMathOperator{\diam}{\mathrm{diam}}

\DeclareMathOperator{\per}{per}

\DeclareMathOperator{\area}{\mathrm{area}}
\DeclareMathOperator{\vol}{\mathrm{vol}}

\makeatletter
\DeclareFontFamily{U}{tipa}{}
\DeclareFontShape{U}{tipa}{m}{n}{<->tipa10}{}
\newcommand{\arc@char}{{\usefont{U}{tipa}{m}{n}\symbol{62}}}%

\newcommand{\arc}[1]{\mathpalette\arc@arc{#1}}

\newcommand{\arc@arc}[2]{%
  \sbox0{$\m@th#1#2$}%
  \vbox{
    \hbox{\resizebox{\wd0}{\height}{\arc@char}}
    \nointerlineskip
    \box0
  }%
}
\makeatother

\newcommand{\adam}[1]{{#1}}
\newcommand{\feri}[1]{{#1}}

\title[Stability of the ball convex P\'al inequality]{Optimal stability of P\'al's isominwidth inequality for ball convex bodies in planes of constant curvature}

\author[F. Fodor]{Ferenc Fodor$^{\orcidlink{0000-0001-9747-1981}}$}

\address{Bolyai Institute, University of Szeged, Aradi v\'ertan\'uk tere 1, 6720 Szeged, Hungary}

\email{fodorf@math.u-szeged.hu}

\author[\'A. Sagmeister]{Ádám Sagmeister$^{\orcidlink{0000-0002-6863-6481}}$}

\address{\parbox{\linewidth}{Bolyai Institute, University of Szeged, Aradi v\'ertan\'uk tere 1, 6720 Szeged, Hungary\\HUN-REN Alfréd R\'enyi Institute of Mathematics, Re\'altanoda u. 13-15, 1053 Budapest, Hungary}}

\email{sagmeister@server.math.u-szeged.hu, sagmeister.adam@gmail.com}


\date{\today}

\begin{document}

\begin{abstract}
    P\'al's isominwidth inequality (1921) answered the Kakeya needle problem (1917) for convex sets. It states that among convex bodies of fixed minimum width $w$ in the Euclidean plane, the regular triangle has minimal area. The isominwidth inequality was generalized to the $2$-dimensional sphere by Bezdek and Blekherman \cite{Bez00} and Freyer and Sagmeister \cite{FS}. Interestingly, in hyperbolic space, no minimizer exists, as shown by B\"or\"oczky, Freyer and Sagmeister \cite{BoFS}. The stability of the Euclidean Pál inequality with respect to the Hausdorff metric and the symmetric difference metric was proved by Lucardesi and Zucco \cite{LZ24}. Fodor, Robock and Sagmeister \cite{FRS26} proved  $r$-ball convex analogs of the isominwidth inequality in all three constant curvature planes connecting Pál's theorem with the Blaschke--Lebesgue inequality. In this paper, we prove optimal stability versions of this statement with respect to the Hausdorff distance and the symmetric difference metric in all three constant curvature planes. 
\end{abstract}

\maketitle

\section{Introduction}
\label{sec:intro}

The concept of width plays an important role in convex geometry. The width of a convex body $K$ (compact convex set whose interior is nonempty) in $\R^n$ in a given direction $u$ is the distance between its two supporting hyperplanes orthogonal to $u$. The body $K$ has constant width if its width is the same in every direction. The theory of bodies of constant width is a rich topic whose results have a wide range of applications; for more information, we refer to the book by Martini, Montejano and Oliveros \cite{MMO19}. One of the most well-known results about bodies of constant width is the Blaschke--Lebesgue theorem (Blaschke \cite{Bla15}, Lebesgue \cite{Leb14}), which states that the Reuleaux triangle has minimal area among plane convex bodies of fixed constant width $w>0$.  The Blaschke--Lebesgue theorem was extended to the $2$-dimensional sphere $\sph^2$ by Ara\'ujo \cite{A97} and Leichtweiss \cite{Lei05} under some smoothness conditions on the boundary. More recently, K. Bezdek \cite{Bez21} gave a proof in the general case when $w<\frac{\pi}{2}$, and Freyer, Sagmeister \cite{FS} proved the inequality when $w\geq\frac{\pi}{2}$. In \cite{FS}, the equality case was also characterized.
B\"or\"oczky and Sagmeister \cite{BoS22} proved the Blaschke--Lebesgue theorem on the hyperbolic plane $\hyp^2$. 
In the hyperbolic plane the Reuleaux triangle is the area minimizer. In the sphere, the Reuleaux triangle still has a minimal area if the minimal width is less than $\frac{\pi}2$. 
However, if the minimal width exceeds $\frac{\pi}2$, then the minimizer is the polar of a Reuleaux triangle. We note that the Blaschke--Lebesgue problem is open in higher dimensions. Campi, Colesanti and Gronchi \cite{CCG96} proved that, among $3$-dimensional rotationally symmetric convex bodies of constant width, the body that one obtains by rotating a Reuleaux triangle about one of its axes of symmetry minimizes the volume. 
Anciaux and Guilfoyle \cite{AB11} proved some necessary conditions on the boundary of extremal bodies in $\R^3$ with respect to the Blaschke--Lebesgue problem. Anciaux and Georgiou \cite{AG26} have recently established that, among bodies of revolution in $\R^3$ with constant width, the ratio of the volume to the cube of the width is minimal for the rotated copy of a Reuleaux triangle. We refer for further references and for an overview of problems of similar type in $\R^3$ to \cite{CCG96}.

In 1921, P\'al \cite{Pal} proved a theorem using the minimal width that is similar in taste to the Blaschke--Lebesgue theorem. The statement is also referred to as the isominwidth theorem after the paper of Ca\~nete and Gonz\'alez Merino \cite{CM21}, and states that the regular triangle has minimal area among plane convex bodies of a given minimal width $w>0$. The original motivation for P\'al's result was to answer the Kakeya needle problem (1917) for convex sets for which the isominwidth theorem provides the solution.
Although no higher dimensional analog of the isominwidth theorem is known, it was extended to the $2$-dimensional sphere. Similarly to the \adam{Euclidean P\'al} theorem, the minimizer is the \feri{regular spherical} triangle if the minimal width is less than $\pi/2$ (K. Bezdek and Blekherman \cite{Bez00}). If the minimal width exceeds $\pi/2$, then the minimizer is the polar of a Reuleaux triangle (Freyer, Sagmeister \cite{FS}). Interestingly, no minimizer exists in hyperbolic space, as shown by B\"or\"oczky, Freyer and Sagmeister \cite{BoFS}. However, in $\hyp^2$, if the problem is restricted to \emph{horocyclically convex} bodies (also called \emph{horoconvex} or \emph{h-convex}, that is, the intersection of horoballs, which are Euclidean balls touching the unit ball in the Poincar\'e ball model), then the unique minimizer of the area with respect to the width introduced by Lassak \cite{Las23} is the regular horocyclic triangle; see B\"or\"oczky, Freyer, Sagmeister \cite{BoFS}. P\'al's problem is also open in higher dimensions. The only $3$-dimensional result that we are aware of is due to Campi, Colesanti, Gronchi \cite{CCG96}, who proved that among rotationally symmetric convex bodies in $\R^3$ whose minimal width is fixed, the rotation of the regular triangle about one of its axes of symmetry minimizes the volume.  

Since the unique extremizer is known both in the (planar) Blaschke--Lebesgue and P\'al theorems, the natural question of stability arises.  An essentially optimal stability version of the Blaschke--Lebesgue inequality with respect to the Hausdorff metric was proved by B\"or\"oczky and Sagmeister \cite{BoS22} in all three constant curvature planes. The stability of the Euclidean Pál inequality with respect to the Hausdorff metric and the symmetric difference metric was proved by Lucardesi and Zucco \cite{LZ24}. Their results are not only of optimal order, but they have also determined the involved constants. The stability of the horocyclically convex case of P\'al's inequality in $\hyp^2$ was proved by B\"or\"oczky, Freyer, Sagmeister \cite{BoFS}, although with a non-optimal order of magnitude.  Freyer and Sagmeister \cite{FS} obtained stability versions of the spherical P\'al inequality both in the case when $w\leq \frac{\pi}2$ and $w>\frac{\pi}{2}$. While the order is optimal in the first case, it is not optimal in the second case.

A further way to generalize the above statements is to consider the so-called $r$-ball convex bodies. A convex body is $r$-ball convex if it is the intersection of closed balls of radius $r$. The concept of $r$-ball convexity provides a unified framework for the Blaschke--Lebesgue and P\'al problems. For a fixed minimal width $w$, let $r\geq w$ in $\R^2$ and $\hyp^2$, and let $w\leq r<\frac{\pi}{2}$ in $\sph^2$. In $\R^2$, the case $r=w$ provides the Blaschke--Lebesgue inequality, and in the limit as $r\to\infty$ we obtain P\'al's problem. In $\sph^2$, the limit case $r\to\frac{\pi}{2}$ is the spherical version of the P\'al-inequality. In $\hyp^2$, the horocyclic P\'al inequality is obtained in the limit as $r\to\infty$. The $r$-ball convex version of P\'al's isominwidth inequality was proved by M. Bezdek \cite{Bez09} for disk polygons (intersection of a finite number of equal radius closed circular disks) in $\R^2$. The general version in all three constant curvature planes was proved by Fodor, Robock and Sagmeister \cite{FRS26}. The precise formulation is provided in the following theorem.

\begin{theorem}[\cite{FRS26}]
\label{thm:isominwidth:spindleconvex}
Let $K$ be an $r$-ball convex body in $\M^2$ where $\M^2$ stands for one of $\R^2$, $\hyp^2$ or $\sph^2$. Let $w$ denote the minimal width of $K$, and let $w\leq r$ (if $\M^2=\sph^2$, we also assume $r<\frac{\pi}{2}$). Then, we have
$$
\area\left(K\right)\geq\area\left(T_{w,r}\right),
$$
where $T_{w,r}$ is the $r$-ball convex hull of a regular triangle with the same minimal width as $K$. We have equality if and only if $K$ is congruent to $T_{w,r}$.
\end{theorem}

In this paper, we prove optimal stability versions of Theorem~\ref{thm:isominwidth:spindleconvex} with respect to the Hausdorff distance $\delta(\cdot,\cdot)$ and the symmetric difference $\Delta(\cdot,\cdot)$; for precise definitions, see \eqref{def:Hausdorff} and \eqref{def:symmetric-difference}. Our main result is the following stability statement.

\begin{theorem}\label{thm:stab:isominwidth:spindleconvex}
Let $K$ be an $r$-ball convex body in $\M^2$ of minimal width $w$, and let $w\leq r$ (in the case $\M^2=\sph^2$ we also assume $r<\frac{\pi}{2}$). Assume
$$
\area\left(K\right)\leq\area\left(T_{w,r}\right)+\varepsilon
$$
for any $\varepsilon\in(0,\overline{\varepsilon})$ where $\overline{\varepsilon}$ is sufficiently small. Then, there are a positive constant $c_\delta$ and an isometry $\phi\colon\M^2\to\M^2$ such that
$$
\delta(K,\phi T_{w,r})\leq c_\delta\varepsilon.
$$
\end{theorem}


Based on Groemer's ideas \cite{G00}, we prove (in Section~\ref{sec:symmetric-difference}) the following estimates that connect the Hausdorff distance and the symmetric difference on $\sph^2$ and $\hyp^2$.

\begin{lemma}\label{lem:groemer-noneuclidean-taylor}
Let $K$ and $L$ be convex bodies in $\M^2$. Let $D=\max\{\diam K,\diam L\}$. If $\M^2=\sph^2$, then assume also that $K\cup L$ is contained in an open hemisphere. Then there exists an $\tilde{\varepsilon}>0$ such that if $\delta(K,L)<\tilde{\varepsilon}$, then
$\Delta(K,L)\leq c_G\cdot\delta (K,L)$ for some constant $c_G$ that depends only on $D$.
\end{lemma}

Lemma~\ref{lem:groemer-noneuclidean-taylor} yields the following corollary.

\begin{theorem}
\label{thm:stab:symmetric-difference}
    Under the same hypotheses as in Theorem~\ref{thm:stab:isominwidth:spindleconvex}, it holds that
 there are a positive constant $c_\Delta$ and an isometry $\phi\colon\M^2\to\M^2$ such that
$$
\Delta(K,\phi T_{w,r})\leq c_\Delta\varepsilon.
$$ 
\end{theorem}

We note that the arguments leading to Theorems~\ref{thm:stab:isominwidth:spindleconvex} and \ref{thm:stab:symmetric-difference} also work in the limit cases, that is, when $r\to \frac{\pi}{2}$ on $\sph^2$ and when $r\to\infty$ in $\hyp^2$. Thus, Theorem~1.2 of Freyer and Sagmeister \cite{FS} and Lemma~\ref{lem:groemer-noneuclidean-taylor} yield the following stability statement for the symmetric difference on $\sph^2$, which is the spherical analog of the corresponding statement of Lucardesi and Zucco \cite{LZ24}.
\begin{theorem}\label{thm:spherical-symmetric-differene}
Let $T_w\subset\sph^2$ denote a regular triangle of width $0<w\leq\frac{\pi}{2}$. It holds that there are a positive constant $c_\Delta$ and an isometry $\phi\colon\sph^2\to\sph^2$ such that
$$
\Delta(K,\phi T_w)\leq c_\Delta\varepsilon
$$
for any $\varepsilon\in(0,\overline{\varepsilon})$ and all convex bodies $K\subset\sph^2$ with $\area(K)\leq\area(T_w) + \varepsilon$.
\end{theorem}

We also note that in the limit case $r\to\infty$ in $\hyp^2$, one obtains the horocyclically convex P\'al inequality, thus, Theorem~\ref{thm:isominwidth:spindleconvex}
improves on the orders of magnitude of the previously non-optimal stability result of B\"or\"oczky, Freyer, Sagmeister \cite{BoFS}, and we also obtain a stability statement for the symmetric difference.

\begin{theorem}\label{thm:hyperbolic-limit-cases}
For $w>0$, let $T_w$ denote the regular horocyclic triangle of minimal width $w$. If $K\subset\hyp^2$ is a horocyclically convex body of minimal width $w$ and $\area(K)\leq \area(T_w)+\varepsilon$ for $\varepsilon\in(0,\overline{\varepsilon})$, then there exists an isometry $\phi\colon\hyp^2\to\hyp^2$ such that
$$
\delta(K,\phi T_w)\leq c_\delta\varepsilon
$$
and
$$
\Delta(K,\phi T_w)\leq c_\Delta\varepsilon
$$
for some positive constants $c_\delta$ and $c\Delta$.
\end{theorem}

Finally, we note that optimality holds in the limit cases as well.

The remainder of the paper is organized as follows. We gather some necessary tools and information in Section~\ref{sec:preliminaries}. Theorem~\ref{thm:stab:isominwidth:spindleconvex} is proved in Section~\ref{sec:stability}.  We prove Lemma~\ref{lem:groemer-noneuclidean-taylor} in Section~\ref{sec:symmetric-difference} that directly yields Theorem~\ref{thm:stab:symmetric-difference}. Finally, the optimality of Theorems~\ref{thm:stab:isominwidth:spindleconvex} and \ref{thm:stab:symmetric-difference} is verified in Section~\ref{sec:optimality}.

\section{Preliminaries}\label{sec:preliminaries}

We use the notation $\M^n$ for an $n$-dimensional space of constant curvature with the geodesic metric $d$. Although many of the concepts used in the paper depend on which constant curvature space is considered, we simplify the notation by suppressing this dependence if it is clear from the context. We only indicate the underlying space if the clarity of exposition requires it. This rule is applied to all the quantities that appear throughout the paper.

For a pair of points $x,y\in\M^n$, the unique geodesic segment with endpoints $x$ and $y$ is denoted as $[x,y]$ (where we also assume that $x$ and $y$ are not antipodal if $\M^n=\sph^n$). A set $K\subset\M^n$ is \emph{convex} if (in the case $\M^n=\sph^n$, we assume that $K$ is contained in an open hemisphere and) $[x,y]\subseteq K$ for all $x,y\in K$. We say that $K\subset\M^n$ is a \emph{convex body} if it is convex, compact, and has a non-empty interior.

The distance of a point $x$ and a compact set $A$ in $\M^n$ is 
\[
d(x, A)=\inf_{a\in A}d(x,a).
\]
The Hausdorff distance of non-empty compact sets $A, B\subset \M^n$ is defined as
\begin{equation}\label{def:Hausdorff}
    \delta(A,B)=\max \{\sup_{a\in A}d(a,B), \sup_{b\in B}d(b,A)\}.
\end{equation}
For properties of Hausdorff-distance, we refer to \cite{Sch14}. For $\lambda>0$ (we assume that $\lambda<\pi/2$ if $\M^n=\sph^n$), the radius $\lambda$ parallel domain of a set $A$ is
\[
A^{(\lambda)}=\{x\in \M^n\colon d(x,A)\leq \lambda\}.
\]
Then for compact sets $A$ and $B$, 
\[
\delta (A,B)=\min\{\lambda\geq 0\colon A\subset B^{(\lambda)}\text{ and } B\subset A^{(\lambda)}\}.
\]

We use $\vol_n(\cdot)$ to denote the volume (Lebesgue measure) of (measurable) sets in $\M^n$. If $n=2$, then $\area(\cdot)$ is often used for the area instead of $\vol_2(\cdot)$. For $A, B\subset \M^n$, the symmetric difference distance is defined as 
\begin{equation}\label{def:symmetric-difference}
    \Delta (A, B)=\vol_n (A\cup B)-\vol_n (A\cap B).
\end{equation}

Groemer proved bounds between the Hausdorff metric and the symmetric difference metric in $\R^n$; see \cite{G00}*{(1)--(4) on p. 108}. For our purposes, the following inequality is important (this is (1) in \cite{G00}): Let $K$ and $L$ be compact convex sets in $\R^n$ with $K\cap L\neq\emptyset$, and let $D=\max\{\diam K, \diam L\}$, where $\diam$ denotes the diameter of a set. Then
\begin{equation}\label{eq:groemer}
   \Delta(K, L)\leq c_G \delta (K, L),  
\end{equation}
with $c_G=(2\kappa_n/(2^{1/n}-1))(D/2)^{n-1}$, where $\kappa_n$ is the volume of a unit ball of $\R^n$.

We are unaware of $n$-dimensional spherical and hyperbolic analogs of this inequality (and, in general, those in \cite{G00}). However, we prove and use the analogs of \eqref{eq:groemer} stated in Lemma~\ref{lem:noneuclidean-groemer} for $\sph^2$ and $\hyp^2$.

The notion of width depends on the space. 
The \emph{width} of a convex body $K$ with respect to a supporting hyperplane $H$, denoted by $w(K,H)$, is the geodesic distance of $H$ and a farthest parallel or ultraparallel supporting hyperplane to $K$ if $\M^n=\R^n$ or $\M^n=\hyp^n$. In $\sph^n$, $w(K,H)$ is defined with lunes. A \emph{lune} is the intersection of two different hemispheres with non-antipodal centers. The \emph{corner} of a lune is the intersection of its two bounding hyperplanes. A lune $L$ \emph{supports} a convex body $K$ if $K\subset L$ and $K$ intersects both hyperplanes bounding $L$ outside its corner. The \emph{breadth} of the lune $L$ is the angle of the two hyperplanes bounding $L$, and $w(K,H)$ is the minimum breadth of all supporting lunes to $K$ such that one of their bounding hyperplanes is $H$. In all cases, the width function is continuous in $K$ with respect to the Hausdorff metric and in $H$ with respect to the topology of the sphere; for a detailed discussion of the hyperbolic case, see B\"or\"oczky, Cs\'epai and Sagmeister \cite{BoCsS}*{Theorem~1.1} and Lassak \cite{La15} for the spherical case. We note that the hyperbolic width functions studied in \cite{BoCsS} are different from the one we are using in this paper, however, this notion of width coincides with the so-called ``extended Leichtweiss width'' introduced in \cite{BoCsS} restricted to supporting hyperplanes (see Lassak \cite{Las23}*{Proposition~2}). For a fixed convex body $K$, the maximum of the width function among all supporting hyperplanes is the diameter of $K$, and the minimum width of $K$, which is denoted by $w(K)$, is monotone with respect to containment (see \cite{La15} and \cite{Las23}).

In this paper, we consider $r$-ball convex bodies. We use the notation $B(x,r)$ for the closed ball of radius $r$ centered at $x$ with respect to the geodesic metric in $\M^n$. In $\sph^n$, $B(x,r)$ is convex if and only if $r<\tfrac{\pi}{2}$. We say that $K\subset\M^n$ is \emph{$r$-ball convex} if it is equal to the intersection of all closed balls of radius $r$ in which it is contained, and if $r<\tfrac{\pi}{2}$ if $\M^n=\sph^n$. Unless $K$ is the empty set or a singleton, the interior of an $r$-ball convex $K$ is non-empty, so $r$-ball convex sets with at least two points are convex bodies. The smallest $r$-ball convex set containing a subset $X$ of some $r$-ball is called the \emph{$r$-ball convex hull} of $X$. If $n=2$, we use the term \emph{$r$-disk} for $B(x,r)$.
The intersections of finitely many $r$-disks in $\M^2$ are called \emph{$r$-disk polygons}. It is not difficult to see that the boundary of an $r$-disk polygon $P$ is the union of a finite number of circular arcs of radius $r$ and, unless $P$ is a disk of radius $r$, it has an equal number of arcs, called \emph{sides}, and non-smooth points, called \emph{vertices}. We introduce the notation $T_{w,r}$ as the $r$-disk triangle with minimal width $w$ where the centers of the three defining $r$-disks are the vertices of a regular triangle in $\M^2$. With the assumption $0<w\leq r$, $T_{w,r}$ is unique up to isometry, and its minimal width is attained by the supporting line at the midpoints of the sides (see Fodor, Robock, Sagmeister \cite{FRS26}).

For a convex body $K\subset\M^n$, let
$$
\varrho(K)=\max\{R>0\colon B(x,R)\subseteq K\text{ for some }x\in K\}.
$$
We call $\varrho(K)$ the \emph{inradius} of $K$ and a ball $B(x,\varrho(K))\subseteq K$ an \emph{inscribed ball} (or an \emph{incircle} if $n=2$) of $K$. Inscribed balls of convex bodies are not unique in general, however, they are if $\M^n=\sph^n$, as well as for horocyclically convex bodies in $\hyp^n$ (see B\"or\"oczky, Freyer, Sagmeister \cite{BoFS}) and ball-convex bodies in any space of constant curvature \cite{FRS26}. 

We will call the domain bounded by a concave circular arc and two convex circular $r$-arcs touching the circle with the concave arc of the domain on its boundary a \emph{cap}. The \emph{apex} of a cap is the intersection point of the two convex $r$-arcs of the cap. A \emph{cap-domain} is the union of a circle and finitely many caps (at least one). We have the following lemma from Fodor, Robock, Sagmeister \cite{FRS26}.

\begin{lemma}\label{lem:nonoverlapping}
Let $K\subset\M^2$ be an $r$-ball convex body of minimal width $w$ where $0<w\leq r$ and if $\M^2=\sph^2$, then $r<\frac{\pi}{2}$. Let $B=B(p,\varrho)$ be the incircle of $K$ with $\varrho<\frac{w}{2}$. Then, there is an $r$-ball convex cap-domain $B\cup C_1\cup C_2\cup C_3$ contained in $K$ such that $C_1,C_2,C_3$ are non-overlapping with apexes $q_1,q_2,q_3$ with $d(p,q_i)=w-\varrho$.
\end{lemma}

We also have the following formula for the inradius of $r$-ball convex regular triangles of width $w$ \cite{FRS26}.

\begin{lemma}\label{lem:rho0}
Let $T_{w,r}\subset\M^2$ be an $r$-ball convex regular triangle of minimal width $w$ where $0<w\leq r$ and if $\M^2=\sph^2$, we also assume $r<\frac{\pi}{2}$, and let $\varrho_0$ be the inradius of $T_{w,r}$. Then
\begin{equation}
\varrho_0=
\begin{cases}
\frac{r+w-\sqrt{\tfrac{4r^2-(r-w)^2}{3}}}{2}, &\text{if }\M^2=\R^2,\\
\frac{r+w-\arcosh\tfrac{4\cosh r-\cosh(r-w)}{3}}{2}, &\text{if }\M^2=\hyp^2,\\
\frac{r+w-\arccos\tfrac{4\cos r-\cos(r-w)}{3}}{2}, &\text{if }\M^2=\sph^2.
\end{cases}
\end{equation}
\end{lemma}

For $r$-ball convex bodies of minimal width $w$ in $\M^2$, we have the following bound on the inradius (see Fodor, Robock, Sagmeister \cite{FRS26}).

\begin{theorem}\label{thm:Blaschke}
Let $K\subset\M^2$ be an $r$-ball convex body of minimal width $w$ where $0<w\leq r$ and if $\M^2=\sph^2$, we also assume $r<\frac{\pi}{2}$. Let $T_{w,r}\subset\M^2$ denote an $r$-ball convex regular triangle of minimal width $w$. Then
$$
\varrho\left(K\right)\geq\varrho\left(T_{w,r}\right)
$$
with equality if and only if $K$ and $T_{w,r}$ are congruent.
\end{theorem}

For our main result, we will need the area formula for triangles in $\M^2$ using two sides and the enclosed angle. In $\hyp^2$ and in $\sph^2$ this is less known, so we cite it below. For the proof of the hyperbolic version, visit the PhD thesis of Frenkel \cite{Fre18}, while the proof of the spherical formula can be found in the PhD thesis of Sagmeister \cite{phd}.

\begin{lemma}\label{lemma:triangle-area}
Let $T\subset\M^2$ be a triangle of side lengths $a,b,c$ and opposite angles $\alpha,\beta,\gamma$. Then for the area of $T$ we have the following equations depending on $\M^2$:
$$\cot\left(\frac{\area\left(T\right)}{2}\right)=\frac{\coth\frac{a}{2}\coth\frac{b}{2}-\cos\gamma}{\sin\gamma}\text{ if }\M^2=\hyp^2,$$
$$\cot\left(\frac{\area\left(T\right)}{2}\right)=\frac{\cot\frac{a}{2}\cot\frac{b}{2}+\cos\gamma}{\sin\gamma}\text{ if }\M^2=\sph^2.$$
\end{lemma}

We also need the following area formula for circular segments in $\M^2$.

\begin{lemma}\label{lem:area-of-circular-segment}
Let $t\in(0,2r)$. Then, the area of a circular segment $S_0(t)$ of an $r$-disk with chord length $t$ is
\[
\area(S_0(t))=\begin{cases}
    \frac{r^2}{2}(\gamma(t)-\sin\gamma(t)),&\text{if }\M^2=\R^2,\\
    (\cosh r-1)\gamma(t)-2\arccot\frac{\coth^2\frac{r}{2}-\cos\gamma(t)}{\sin\gamma(t)},&\text{if }\M^2=\hyp^2,\\
    (1-\cos r)\gamma(t)-2\arccot\frac{\cot^2\frac{r}{2}+\cos\gamma(t)}{\sin\gamma(t)},&\text{if }\M^2=\sph^2.
\end{cases}
\]
with
\[
\gamma(t)=\begin{cases}
    \arccos\Big(1-\frac{t^2}{2r^2}\Big),&\text{if }\M^2=\R^2,\\
    \arccos\Big(\frac{\cosh^2 r-\cosh^2 t}{\sinh^2 t}\Big),&\text{if }\M^2=\hyp^2,\\
    \arccos\Big(\frac{\cos^2 t-\cos^2 r}{\sin^2 t}\Big),&\text{if }\M^2=\sph^2.
\end{cases}
\]
\end{lemma}

\begin{proof}
We can calculate $\area(S_0(t))$ as the difference of the area of a circular sector of an $r$-disk with a chord of length $t$ and an isosceles triangle with base length $t$ and leg length $r$. We obtain the central angle of the circular sector $\gamma(t)$ from the law of cosines in the isosceles triangle. The circular sector with central angle $\gamma(t)$ has area
\[
\begin{cases}
\frac{r^2}{2}\gamma(t),&\text{if }\M^2=\R^2,\\
(\cosh r-1)\gamma(t),&\text{if }\M^2=\hyp^2,\\
(1-\cos r)\gamma(t),&\text{if }\M^2=\sph^2.
\end{cases}
\]
The area of the isosceles triangle with base $t$ and legs of length $r$ can be obtained from Lemma~\ref{lemma:triangle-area}.
\end{proof}

The idea to prove Theorem~\ref{thm:isominwidth:spindleconvex} is to compare the areas of cap-domains with three congruent non-overlapping caps. If the inradius of the $r$-ball convex body $\varrho$ is smaller than $\frac{w}{2}$, then any $r$-ball convex body in $\M^2$ contains such a cap-domain that can be rearranged in a rotationally symmetric way. In this symmetric cap-domain, we can find an $r$-disk hexagon $Q_{w,r,\varrho}$ (see Figure~\ref{fig:Qwrrho}), whose area is greater than that of the $r$-disk triangle $T_{w,r}$. We also need this result for the stability of the inequality. For a proof, visit Fodor, Robock, Sagmeister \cite{FRS26}.

\begin{lemma}\label{lem:areaofhexagonsincreases}
Let $B\subset\M^2$ be a circular disk centered at $p$ and of radius $\varrho\in\left[\varrho_0,\frac{w}{2}\right)$ where $\varrho_0$ denotes the inradius of $T_{w,r}$. Let $q_1,q_2,q_3$ be points of $\M^2$ such that $d\left(p,q_i\right)=w-\varrho$ and $\angle\left(q_i,p,q_j\right)=\frac{2\pi}{3}$ for $i\neq j$. We also introduce the points $t_i$ as the intersection of the half-line emanating from $q_i$ through $p_i$ and $\partial B$. Finally, we set $Q_{w,r,\varrho}$ to be the $r$-ball convex hull of $q_1,q_2,q_3,t_1,t_2,t_3$. Then, the area of $Q_{w,r,\varrho}$ is greater than the area of $T_{w,r}$.
\end{lemma}

\begin{figure}[H]
\centering
\begin{tikzpicture}[line cap=round,line join=round,>=triangle 45,x=1cm,y=1cm, scale=4.5]
\clip(-0.072,-0.22) rectangle (1.144,0.9208701508801388);
\draw [shift={(1,0)},line width=.8pt,dotted]  plot[domain=2.0943951023931953:3.141592653589793,variable=\t]({1*1*cos(\t r)+0*1*sin(\t r)},{0*1*cos(\t r)+1*1*sin(\t r)});
\draw [shift={(0,0)},line width=.8pt,dotted]  plot[domain=0:1.0471975511965976,variable=\t]({1*1*cos(\t r)+0*1*sin(\t r)},{0*1*cos(\t r)+1*1*sin(\t r)});
\draw [shift={(0.5,0.8660254037844388)},line width=.8pt,dotted]  plot[domain=4.1887902047863905:5.235987755982988,variable=\t]({1*1*cos(\t r)+0*1*sin(\t r)},{0*1*cos(\t r)+1*1*sin(\t r)});
\draw [shift={(0.9246504784333377,-0.06668223747786052)},line width=2pt]  plot[domain=2.0093721182089483:2.5224251945822695,variable=\t]({1*1*cos(\t r)+0*1*sin(\t r)},{0*1*cos(\t r)+1*1*sin(\t r)});
\draw [shift={(0.5954232724203452,0.834111922683636)},line width=2pt]  plot[domain=4.103767220602143:4.616820296975464,variable=\t]({1*1*cos(\t r)+0*1*sin(\t r)},{0*1*cos(\t r)+1*1*sin(\t r)});
\draw [shift={(1.020073750853683,0.09859571857866334)},line width=2pt]  plot[domain=2.7135625614007193:3.2266156377740405,variable=\t]({1*1*cos(\t r)+0*1*sin(\t r)},{0*1*cos(\t r)+1*1*sin(\t r)});
\draw [shift={(0.4045767275796549,0.8341119226836362)},line width=2pt]  plot[domain=4.807957663793914:5.321010740167235,variable=\t]({1*1*cos(\t r)+0*1*sin(\t r)},{0*1*cos(\t r)+1*1*sin(\t r)});
\draw [shift={(-0.02007375085368288,0.09859571857866313)},line width=2pt]  plot[domain=-0.08502298418424736:0.4280300921890742,variable=\t]({1*1*cos(\t r)+0*1*sin(\t r)},{0*1*cos(\t r)+1*1*sin(\t r)});
\draw [shift={(0.07534952156666211,-0.06668223747786037)},line width=2pt]  plot[domain=0.6191674590075236:1.1322205353808448,variable=\t]({1*1*cos(\t r)+0*1*sin(\t r)},{0*1*cos(\t r)+1*1*sin(\t r)});
\begin{scriptsize}
\draw [fill=black] (0,0) circle (.2pt);
\draw[color=black] (-0.015,-0.025) node {$v_2$};
\draw [fill=black] (1,0) circle (.2pt);
\draw[color=black] (1.015,-0.025) node {$v_3$};
\draw [fill=black] (0.5,0.8660254037844388) circle (.2pt);
\draw[color=black] (0.5,0.89) node {$v_1$};
\draw [fill=black] (0.5,0.2886751345948129) circle (.2pt);
\draw[color=black] (0.5,0.31614332483983393) node {$p$};
\draw [fill=black] (0.5,0.838675134594813) circle (.2pt);
\draw[color=black] (0.5,0.81) node {$q_1$};
\draw [fill=black] (0.5,-0.1613248654051871) circle (.2pt);
\draw[color=black] (0.5,-0.189) node {$t_1$};
\draw [fill=black] (0.9763139720814413,0.013675134594812843) circle (.2pt);
\draw[color=black] (0.955,0.038) node {$q_3$};
\draw [fill=black] (0.11028856829700262,0.513675134594813) circle (.2pt);
\draw[color=black] (0.1,0.54) node {$t_3$};
\draw [fill=black] (0.023686027918558724,0.013675134594812873) circle (.2pt);
\draw[color=black] (0.045,0.038) node {$q_2$};
\draw [fill=black] (0.8897114317029974,0.513675134594813) circle (.2pt);
\draw[color=black] (0.9,0.54) node {$t_2$};
\end{scriptsize}
\end{tikzpicture}
\caption{The $r$-disk hexagon $Q_{w,r,\varrho}$ with thick boundary and the $r$-disk triangle $T_{w,r}$ with dotted boundary}
\label{fig:Qwrrho}
\end{figure}
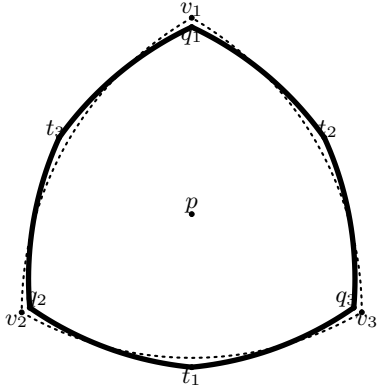

\adam{In fact, one can strengthen the statement of the previous lemma as follows. We will use the following stronger version to verify Lemma~\ref{lem:boundingeta}.

\begin{lemma}\label{lem:monotonicity}
With the assumptions and notations of Lemma~\ref{lem:areaofhexagonsincreases}, $\area(Q_{w,r,\varrho})$ is a strictly monotone increasing function of $\varrho$ on the interval $[\varrho_0,w/2]$.
\end{lemma}

The proof of this lemma is rather technical, and we only use it at one point of the paper, so for the sake of readability, we defer the proof to the appendix.
}

\section{Stability with respect to the Hausdorff metric}\label{sec:stability}

In this section, we verify our main result, Theorem~\ref{thm:stab:isominwidth:spindleconvex}. We set $\eta=\varrho(K)-\varrho_0$ and \adam{
\begin{equation}\label{bar-epsilon}
\overline{\varepsilon}=\begin{cases}
    \area\left(B\left(p,\frac{w}{2}\right)\right)-\area(T_{w,r}), & \text{if }\M^2=\R^2,\\
    \min\left\{\area\left(B\left(p,\frac{w}{2}\right)\right)-\area(T_{w,r}),\frac{\tanh\frac{\varrho_0}{2}\sin\alpha_0}{4}\right\}, & \text{if }\M^2=\hyp^2,\\
    \min\left\{\area\left(B\left(p,\frac{w}{2}\right)\right)-\area(T_{w,r}),\left(\tan\frac{\varrho_0}{2}\sin\alpha_0\right)\pi,1\right\}, & \text{if }\M^2=\sph^2.
\end{cases}
\end{equation}
where $\alpha_0$ is defined in \eqref{eq:alpha0}.}
To simplify the notation, we denote the inradius of $K$ by $\varrho$. First, we prove that $\eta$ can be bounded from above by a constant times $\varepsilon$.

\begin{lemma}\label{lem:boundingeta}
Let $K$ be an $r$-ball convex body of minimal width $0<w\leq r$ and $\area(K)\leq\area(T_{w,r})+\varepsilon$ with $\varepsilon\in(0,\overline{\varepsilon})$. Then
$$
\eta=\varrho(K)-\varrho_0\leq c_1\varepsilon
$$
for some positive constant $c_1$.
\end{lemma}
\begin{proof}
We have $\eta\in\left[0,\frac{w}{2}-\varrho_0\right]$ with $0<\frac{w}{2}-\varrho_0$ from Lemma~\ref{lem:rho0} and Theorem~\ref{thm:Blaschke}. Consider the $r$-ball convex hexagon $Q_{w,r,\varrho}$ described in Lemma~\ref{lem:areaofhexagonsincreases}. The area of $Q_{w,r,\varrho}$ can be expressed as a function of $\eta$; let $f(\eta)=\area(Q_{w,r,\varrho})$. To find $f(\eta)$, we can dissect $Q_{w,r,\varrho}$ into six congruent copies of triangles with sides $\varrho$ and $w-\varrho$ and an enclosed angle $\tfrac{\pi}{3}$, and six congruent circular segments of an $r$-disk, where the chord length is the opposite side of the angle $\tfrac{\pi}{3}$ in these triangles and can be derived from the law of cosines. Hence, as a consequence of Lemma~\ref{lemma:triangle-area} and Lemma~\ref{lem:area-of-circular-segment}, $f\colon\left[0,\frac{w}{2}-\varrho_0\right]\to\R$ is a continuously differentiable function that attains its unique minimum at 0 due to Lemma~\ref{lem:areaofhexagonsincreases}. Then, by \cite{FS}*{Lemma 3.5}, there is a positive constant $c_1$ such that
$$
\eta\leq c_1(f(\eta)-f(0))=c_1(\area(Q_{w,r,\varrho})-\area(T_{w,r}))\leq c_1(\area(K)-\area(T_{w,r}))=c_1\varepsilon.
$$
\end{proof}

Since $\varepsilon<\overline{\varepsilon}$, $K$ contains a cap-domain $C_{w,r,\varrho}$ that can be described as the $r$-ball convex hull of the incircle $B=B(p,\varrho)$ of $K$ and the points $q_1,q_2,q_3$ where these three points determine three congruent non-overlapping regions in positive orientation, bounded by circular arcs. We choose $\phi$ to be an isometry such that $\phi(p)$ is the center of $B$ and $q_1\in[\phi(p),\phi(v_1)]$, where $v_1,v_2,v_3$ are the vertices of $T_{w,r}$. We claim the following.

\begin{lemma}\label{lem:symmetriccapdomain-triangle}
Let $Q_{w,r,\varrho}$ be an $r$-ball convex hexagon with vertices $\widetilde{q}_1,\widetilde{t}_3,\widetilde{q}_2,\widetilde{t}_1,\widetilde{q}_3,\widetilde{t}_2$ such that $d(p,\widetilde{q}_i)=w-\varrho$, $d(p,\widetilde{t}_i)=\varrho$ for $i\in\{1,2,3\}$ and $\widetilde{q}_1=q_1$. Let $\widetilde{C}$ be the $r$-ball convex hull of $B$ and the points $\widetilde{q}_1,\widetilde{q}_2,\widetilde{q}_3$. Then
$$
\delta(Q_{w,r,\varrho},\phi(T_{w,r}))=\delta(\widetilde{C},\phi(T_{w,r}))=\eta.
$$
\end{lemma}

\begin{proof}
By the choice of $\phi$, 
$$\delta(Q_{w,r,\varrho},\phi(T_{w,r}))\geq d(q_1,\phi(v_1))=\eta,$$ and 
$$\delta(\widetilde{C},\phi(T_{w,r}))\geq d(q_1,\phi(v_1))=\eta.$$ Also, it is easy to see that $\adam{\phi(T_{w,r})}\subseteq Q_{w,r,\varrho}^{(\eta)}\subseteq\widetilde{C}^{(\eta)}$ and $Q_{w,r,\varrho}\subseteq\widetilde{C}\subseteq \adam{\phi(T_{w,r})}^{(\eta)}$; hence, $\delta(Q_{w,r,\varrho},\phi(T_{w,r}))\leq\eta$ and $\delta(\widetilde{C},\phi(T_{w,r}))\leq\eta$.
\end{proof}

Next, we show that $C$ is close to $\widetilde{C}$. To see this, we notice the following.

\begin{lemma}\label{lem:angles-are-close}
Let $t_1,t_2,t_3$ be points of $\partial K\cap\partial B$ such that $q_1,t_3,q_2,t_1,q_3,t_2$ are in positive cyclic order. For $\{i,j,k\}=\{1,2,3\}$, let $C_i$ be the cap with peak $q_i$, $C_i$ is attached to $B$ at two points $s_{i,j}$ and $s_{i,k}$ such that $t_k$ is on the arc of $\partial B\cap \partial C$ between $s_{i,k}$ and $s_{j,k}$. Then, there is a positive constant $c_2$ such that
\begin{equation}\label{eq:anglebound}
\angle(s_{i,j},p,s_{j,i})\leq c_2\eta.
\end{equation}
\end{lemma}

\begin{proof}
We observe that we can assume that the cap-domain $B\cup C_1\cup C_2\cup C_3$ has a threefold rotational symmetry. Let $c$ denote the center of the circle of radius $r$ containing the arc between $q_1$ and $s_{1,2}$ at the boundary of $C_1$. Since the circle of radius $r$ centered at $c$ is tangent to $B$ in $s_{1,2}$, $p$ is in the relative interior of $[s_{1,2},c]$. The triangle $[q_1,p,c]$ has sides $d(q_1,p)=w-\varrho_0-\eta$, $d(p,c)=r-\varrho_0-\eta$ and $d(c,q_1)=r$, while $\angle(q_1,p,c)=\frac{2\pi}{3}+\frac{1}{2}\angle(s_{1,2},p,s_{2,1})$. Now we can express $\angle(s_{1,2},p,s_{2,1})$ from the law of cosines.
\[
\cos\left(\frac{2\pi}{3}+\frac{1}{2}\angle(s_{1,2},p,s_{2,1})\right)=
\begin{cases}
\frac{(r-\varrho_0-\eta)^2+(w-\varrho_0-\eta)^2-r^2}{2(r-\varrho_0-\eta)(w-\varrho_0-\eta)}, & \text{if }\M^2=\R^2,\\
\frac{\cosh(r-\varrho_0-\eta)\cosh(w-\varrho_0-\eta)-\cosh r}{\sinh(r-\varrho_0-\eta)\sinh(w-\varrho_0-\eta)}, & \text{if }\M^2=\hyp^2,\\
\frac{\cos r-\cos(r-\varrho_0-\eta)\cos(w-\varrho_0-\eta)}{\sin(r-\varrho_0-\eta)\sin(w-\varrho_0-\eta)}, & \text{if }\M^2=\sph^2.
\end{cases}
\]
After rearranging the equation and using Taylor expansion, we obtain
\[
\angle(s_{1,2},p,s_{2,1})=
\begin{cases}
\frac{6\sqrt{3r^2+2rw-w^2}}{2rw-w^2}\eta+O(\eta^2), & \text{if }\M^2=\R^2,\\
\sqrt{\frac{3}{2}}\frac{\sqrt{16\cosh(2r)+\cosh(2(r-w)-8\cosh(2r-w)-8\cosh w-1}}{\cosh r-\cosh(r-w)}\eta+O(\eta^2), & \text{if }\M^2=\hyp^2,\\
\frac{3\sqrt{9-(4\cos r-\cos(r-w))^2}}{4\sin\left(r-\frac{w}{2}\right)\sin\frac{w}{2}}\eta+O(\eta^2), & \text{if }\M^2=\sph^2.
\end{cases}
\]
\end{proof}

\begin{lemma}\label{lem:distance-of-capdomains}
There is a positive constant $c_3$ such that $\delta(C,\widetilde{C})\leq c_3\eta$.
\end{lemma}

\begin{proof}
We notice that $\delta(C,\widetilde{C})\leq\max\{d(q_2,\widetilde{q}_2),d(q_3,\widetilde{q}_3)\}$. We may assume that $d(q_2,\widetilde{q}_2)\geq d(q_3,\widetilde{q}_3)$. We denote the angle $\angle(q_2,p,\widetilde{q}_2)$ by $\varphi$. We can express $d(q_2,\widetilde{q}_2)$ using the law of sines in the right triangle $[p,m,q_2]$, where $m$ is the midpoint of $[q_2,\widetilde{q}_2]$. We know that $d(p,q_2)=w-\varrho_0-\eta$ and $\angle(q_2,p,m)=\frac{\varphi}{2}\leq\frac{c_2}{2}\eta$ from Lemma~\ref{lem:angles-are-close}. We divide the rest of the proof into three cases depending on the plane.

\begin{case}
$\M^2=\R^2$.
\end{case}

In this case,
\[
d(q_2,\widetilde{q}_2)=2(w-\varrho_0-\eta)\sin\frac{\varphi}{2}\leq(w-\varrho_0)\varphi\leq(w-\varrho_0)c_2\eta.
\]
\begin{case}
    $\M^2=\hyp^2$.
\end{case}

On the hyperbolic plane, we have
\[
d(q_2,\widetilde{q}_2)\leq2\sinh\frac{\adam{d(q_2,\widetilde{q}_2)}}{2}=\adam{2}\sinh(w-\varrho_0-\eta)\sin\frac{\varphi}{2}\leq\adam{\sinh(w-\varrho_0)\varphi}\leq\adam{\sinh(w-\varrho_0)c_2\eta}.
\]

\begin{case}
    $\M^2=\sph^2$.
\end{case}

Finally, on the sphere,
\[
\adam{d(q_2,\widetilde{q}_2)}\leq\adam{\frac{4\pi}{3\sqrt{3}}}\sin\frac{\adam{d(q_2,\widetilde{q}_2)}}{2}=\adam{\frac{4\pi}{3\sqrt{3}}}\sin(w-\varrho_0-\eta)\sin\frac{\varphi}{2}\leq\adam{\frac{2\pi}{3\sqrt{3}}\sin(w-\varrho_0)c_2\eta.}
\]
\adam{Here, the first estimate comes from the fact that the function $\frac{\sin x}{x}$ decreases on $(0,\pi)$ and that $\frac\varphi2\leq\frac{2\pi}{3}$.}
\end{proof}

The final ingredient of the proof of Theorem~\ref{thm:stab:isominwidth:spindleconvex} is the bound for the distance between $K$ and the cap-domain $C$.

\begin{lemma}\label{lem:distance-of-K-and-C}
Let $K$ be an $r$-ball convex body of minimal width $0<w\leq r$ and $\area(K)\leq\area(T_{w,r})+\varepsilon$ with $\varepsilon\in(0,\overline{\varepsilon})$. Let $C=B\cup C_1\cup C_2\cup C_3$ be a cap-domain with non-overlapping congruent caps $C_i$ peaks at $q_i$, where $d(q_i,p)=w-\varrho_0-\eta$. Then
$$
\delta(K,C)<c_4\varepsilon
$$
for some positive constant $c_4$.
\end{lemma}

\begin{proof}
Let $K_1,K_2,K_3$ be the regions containing $C_1,C_2,C_3$, respectively, bounded by circular arcs of $\partial B$ and parts of $\partial K$. Let $z$ be a farthest point of $K$ from $p$. We may assume that $z\in K_1$. Moreover, by the ball convexity of $K$, we may assume $q_1\in[p,z]$. Then, it is enough to show that $d(z,q_1)\leq c_4\varepsilon$ for some positive constant $c_4$. If $q_1=z$, then there is nothing to prove, so we assume otherwise. Since $w-\varrho_0-\eta>\varrho_0+\eta$, $d(q_1,s_{1,2})>d(p,s_{1,2})>\varrho_0$ and $\angle(s_{1,3},q_1,s_{1,2})<\angle(s_{1,3},p,s_{1,2})<\frac{2\pi}{3}$. By our assumption $q_1\in[p,z]$, $\angle(s_{1,3},q_1,s_{1,2})+\angle(s_{1,3},q_1,z)=\angle(s_{1,3},q_1,s_{1,2})+\angle(s_{1,2},q_1,z)>\pi$ and $\angle(z,q_1,s_{1,2})=\angle(z,q_1,s_{1,3})>\frac{2\pi}{3}$. On the other hand, $\angle(q_1,z,s_{1,2})<\angle(s_{1,3},z,s_{1,2})\leq\angle(s_{1,3},q_1,s_{1,2})<\frac{2\pi}{3}$. This implies
\begin{equation}\label{eq:sides-of-triangle-zs12q1}
d(s_{1,2},q_1)<d(z,s_{1,2}).
\end{equation}
Since $K$ is $r$-ball convex, the $r$-segment determined by $z$ and $s_{1,2}$ (the $r$-ball convex hull of $z$ and $s_{1,2}$) is contained in $K$. In other words, the region $T$ bounded by $[q_1,z]$, a convex $r$-circular arc connecting $z$ and $s_{1,2}$, and a concave $r$-circular arc connecting $s_{1,2}$ and $q_1$ is contained in $K_1$. From \eqref{eq:sides-of-triangle-zs12q1} and our assumption,
\begin{equation}
\area([z,s_{1,2},q_1])<\area(T)<\varepsilon.
\end{equation}
\adam{We also observe that the angle $\angle(s_{1,2}q_1,z)$ decreases monotonically in $\eta$, so it can be bounded from above by some $\alpha_0\in(\frac{2\pi}{3},\pi)$ depending on $w$, $r$ and $\M^2$. This angle can be easily computed combining the law of cosines and the law of sines:
\begin{equation}\label{eq:alpha0}
\sin\alpha_0=\begin{cases}
    \frac{\sqrt{3}\varrho_0}{2\sqrt{\varrho_0^2+(w-\varrho_0)^2-\varrho_0(w-\varrho_0)}}, & \text{if }\M^2=\R^2,\\
    \frac{\sqrt{3}\sinh\varrho_0}{2\sqrt{\left(\cosh\varrho_0\cosh(w-\varrho_0)-\frac{1}{2}\sinh\varrho_0\sinh(w-\varrho_0)\right)^2-1}}, & \text{if }\M^2=\hyp^2,\\
    \frac{\sqrt{3}\sin\varrho_0}{2\sqrt{1-\left(\cos\varrho_0\cos(w-\varrho_0)+\frac{1}{2}\sin\varrho_0\sin(w-\varrho_0)\right)^2}}, & \text{if }\M^2=\sph^2.
\end{cases}
\end{equation}
One can also observe that its limit as $r$ tends to $\infty$ (or to $\frac{\pi}{2}$ in the spherical case) is also strictly smaller than $\pi$.}
The rest of the proof is divided into cases according to $\M^2$.

\setcounter{case}{0}
\begin{case}
    $\M^2=\R^2$.
\end{case}

In this case,
\[
\varepsilon>\area([z,s_{1,2},q_1])=\frac{d(q_1,s_{1,2})\cdot d(q_1,z)\sin(\angle(s_{1,2},q_1,z))}{2}>\adam{\frac{\sin\alpha_0}{2}}\varrho_0 d(q_1,z),
\]
and hence
\[
d(q_1,z)<\adam{\frac{2}{\varrho_0\sin\alpha_0}}\varepsilon.
\]

\begin{case}
    $\M^2=\hyp^2$.
\end{case}

From the area formula of the hyperbolic triangle in Lemma \ref{lemma:triangle-area},
\[
\varepsilon>2\arccot\left(\frac{\coth\frac{d(s_{1,2},q_1)}{2}\coth\frac{d(z,q_1)}{2}-\cos\angle(s_{1,2},q_1,z)}{\sin\angle(s_{1,2},q_1,z)}\right)>2\arccot\left(\frac{\coth\frac{\varrho_0}{2}\coth\frac{d(z,q_1)}{2}+1}{\adam{\sin\alpha_0}}\right),
\]
which is equivalent to
\[
\tanh\frac{\varrho_0}{2}\left(\adam{\sin\alpha_0}\cot\frac{\varepsilon}{2}-1\right)<\coth\frac{d(z,q_1)}{2}.
\]
We can bound the left-hand side from below, where we also use $\varepsilon\in\left(0,\adam{\frac{\tanh\frac{\varrho_0}{2}\sin\alpha_0}{4}}\right)$.
\[
\tanh\frac{\varrho_0}{2}\left(\adam{\sin\alpha_0}\cot\frac{\varepsilon}{2}-1\right)>\coth\left(\adam{\frac{\coth\frac{\varrho_0}{2}}{\sin\alpha_0}}\varepsilon\right),
\]
and thus
\[
d(q_1,z)<\adam{\frac{2\coth\frac{\varrho_0}{2}}{\sin\alpha_0}}\varepsilon.
\]

\begin{case}
    $\M^2=\sph^2$.
\end{case}

From the area formula in Lemma \ref{lemma:triangle-area}, we have
\[
\varepsilon>2\arccot\left(\frac{\cot\frac{d(s_{1,2},q_1)}{2}\cot\frac{d(z,q_1)}{2}+\cos\angle(s_{1,2},q_1,z)}{\sin\angle(s_{1,2},q_1,z)}\right)>2\arccot\left(\frac{\cot\frac{\varrho_0}{2}\cot\frac{d(z,q_1)}{2}-\frac{1}{2}}{\adam{\sin\alpha_0}}\right),
\]
or equivalently
\[
\tan\frac{\varrho_0}{2}\left(\adam{\sin\alpha_0}\cot\frac{\varepsilon}{2}+\frac{1}{2}\right)<\cot\frac{d(z,q_1)}{2}.
\]
We can bound the left-hand side from below, using $\varepsilon\in(0,1)$ \adam{and $\varepsilon<(\tan\frac{\varrho_0}{2}\sin\alpha_0)\pi$}.
\[
\tan\frac{\varrho_0}{2}\left(\adam{\sin\alpha_0}\cot\frac{\varepsilon}{2}+\frac{1}{2}\right)>\cot\left(\adam{\frac{\cot\frac{\varrho_0}{2}}{\sin\alpha_0}}\varepsilon\right),
\]
and thus
\[
d(q_1,z)<\adam{\frac{2\cot\frac{\varrho_0}{2}}{\sin\alpha_0}}\varepsilon.
\]
\end{proof}

Now we have all auxiliary lemmas to finish the proof.

\begin{proof}[Proof of Theorem~\ref{thm:stab:isominwidth:spindleconvex}]
By the choice of $\overline{\varepsilon}$, $\area(T_{w,r})<\area(K)\leq\area(T_{w,r})+\varepsilon<\area(B(p,\frac{w}{2}))$. Then, by Lemma~\ref{lem:nonoverlapping}, there is a cap-domain $C$ contained in $K$ where $C$ is the $r$-ball convex hull of the incircle $B$ of $K$ and three points $q_1,q_2,q_3$ that add three congruent, pairwise non-overlapping caps $C_1,C_2,C_3$ to $B$. Using Lemma~\ref{lem:boundingeta}, Lemma~\ref{lem:distance-of-capdomains}, Lemma~\ref{lem:distance-of-K-and-C} and the triangle inequality, we obtain
\[
\delta(K,\phi T_{w,r})\leq \delta(K,C)+\delta(C,\widetilde{C})+\delta(\widetilde{C},\phi T_{w,r})\leq c_4\varepsilon+c_3\eta+\eta\leq(c_4+c_1c_3+c_1)\varepsilon.
\]
\end{proof}

\section{Stability with respect to the symmetric difference metric}\label{sec:symmetric-difference}

In this section, we prove Theorem~\ref{thm:stab:symmetric-difference}. We note that for $\M^2=\R^2$, the limit case as $r\to\infty$ of the statement of the theorem was shown by Lucardesi and Zucco \cite{LZ24}. The general statement in $\R^2$ follows directly from Theorem~\ref{thm:stab:isominwidth:spindleconvex} and Groemer's inequality~\eqref{eq:groemer}. For the hyperbolic and spherical case, we use a similar argument to Groemer's \cite{G00}. First, we need the following non-Euclidean planar Steiner formulas (see, e.g., Santal\'o \cite{S04}). For more details on the spherical Steiner formula, see Schneider and Weil \cite{SW08}*{Section~6.5}.

\begin{lemma}\label{lem:steiner}
Let $K$ be a convex body in $\hyp^2$ or in $\sph^2$. For a positive number $\lambda$, the area of the parallel body $K^{(\lambda)}$ of $K$ with radius $\lambda$ (where in $\sph^2$, we also assume that $0\leq \lambda<\pi/2$) is
\begin{equation}
\area (K^{(\lambda)})=\begin{cases}
    \cosh\lambda \area (K)+\sinh\lambda \per (K)+2\pi(\cosh\lambda-1), &\text{if }\M^2=\hyp^2,\\
    \cos\lambda \area (K)+\sin\lambda \per (K)+2\pi(1-\cos\lambda), &\text{if }\M^2=\sph^2.
\end{cases}
\end{equation}
\end{lemma}

We also use the following.

\begin{lemma}\label{lem:maximal-perimeter}
Let $K\subset\M^2$ be a convex body of diameter at most $D$ \feri{with $\M^2=\sph^2$ or $\M^2=\hyp^2$, and let $z\in\M^2$ be an arbitrary point}. Then
\begin{equation}
\per(K)\leq\begin{cases}
    \per \left (B\big(z,\frac{D}{2}\big)\right ), &\text{if }\M^2=\hyp^2,\\
    \per (T), &\text{if }\M^2=\sph^2,
\end{cases}
\end{equation}
with equality if and only if $K$ is congruent to $B\big(z,\frac{D}{2}\big)$ if $\M^2=\R^2$ or $\M^2=\hyp^2$, and $K$ is congruent to $T$ if $\M^2=\sph^2$, where $T$ denotes a regular spherical triangle of diameter (height) $D$ if $D\geq\frac{\pi}{2}$ and a Reuleaux triangle of diameter $D$ if $D<\frac{\pi}{2}$.
\end{lemma}

\begin{proof}
By the monotonicity of the perimeter, it is clear that if a convex body $K$ has a maximal perimeter among convex bodies of diameter at most $D$, then $\diam K=D$.

\setcounter{case}{0}
\begin{case}
$\M^2=\hyp^2$.
\end{case}

Let $K'$ be a completion of $K$ (i.e., a convex set $K'\supset K$ such that $\diam K'=\diam K$ and $\diam K'\cup\{y\}>\diam K'$ for any point $y$). Then $\per (K')\geq\per (K)$ by the monotonicity of the perimeter. 
It is known that complete bodies and bodies of constant width are the same in hyperbolic space (see B\"or\"oczky, Cs\'epai and Sagmeister \cite{BoCsS}), therefore $K'$ has constant width $D$. The hyperbolic version of Barbier's theorem (see, for example, Santal\'o \cite{S45}*{(7.3) on p. 411)} states that
\begin{equation}\label{eq:barbier-hyperbolic}
    \per (K')=(2\pi+\area (K'))\tanh \left (\frac D2\right ).
\end{equation}
We note that Santal\'o \cite{S45} uses a different width function for the hyperbolic version of Barbier's theorem, but for complete bodies both width functions have constant value that is equal to the diameter of the body (see B\"or\"oczky, Cs\'epai and Sagmeister \cite{BoCsS}). Thus, by \eqref{eq:barbier-hyperbolic} the perimeter of $K'$ is maximal if and only if the area of $K$ is maximal. 
According to the hyperbolic isodiametric inequality (see Schmidt \cites{Sch48,Sch49}), the  area of a convex set of diameter $D$ in $\hyp^2$ is maximal precisely when the set is a ball of radius $D/2$. Combining this with \eqref{eq:barbier-hyperbolic}, we get the statement of the lemma. 

\begin{case}
$\M^2=\sph^2$.
\end{case}

It is known (see Bezdek and Blekherman \cite{Bez00}) that the polar $K^\circ$ of a spherical convex body $K$ with diameter $D$ has minimal width $\pi-D$. Furthermore, $\per (K)=2\pi-\area (K^\circ)$. Therefore, $K$ has maximal perimeter with a fixed diameter $D$ if and only if $K^\circ$ has minimal area with a fixed minimal width $\pi-D$. Hence, $K^\circ$ is a regular triangle of minimal width $\pi-D$ if $\pi-D\leq\frac{\pi}{2}$, and $K^\circ$ is the polar of a Reuleaux triangle otherwise (see Freyer and Sagmeister \cite{FS}). This concludes the proof.
\end{proof}

We have the following inequality between the Hausdorff metric and the symmetric difference metric.

\begin{lemma}\label{lem:noneuclidean-groemer}
Let $K$ and $L$ ($K\neq L$) be convex bodies in $\M^2$. Let $D=\max\{\diam K,\diam L\}$. Then
\begin{equation}
\Delta_{\hyp^2}(K,L)<
    4\pi\Big(\cosh\frac{D}{2}(\cosh\delta(K,L)-1)+\sinh\frac{D}{2}\sinh\delta(K,L)\Big)
\end{equation}
and
\begin{equation}
\Delta_{\sph^2}(K,L)<\begin{cases}
    4\pi(1-\cos\delta(K,L))+12\sin\delta(K,L)\sin D \arcsin\frac{1}{2\cos\frac{D}{2}}, &\text{if }D<\frac{\pi}{2},\\
    4\pi(1-\cos\delta(K,L))+12\sin\delta(K,L)\adam{\arccos\frac{\cos D+\sqrt{\cos^2 D+8}}{4}}, &\text{if }\frac{\pi}{2}\leq D<\pi.
\end{cases}
\end{equation}
\end{lemma}

\begin{proof}
We note that $\Delta_{\M^2}(K,L)=\area (K\cup L)-\area (K\cap L)=(\area (K)-\area (K\cap L))+(\area (L)-\area (K\cap L))$. We bound $\area (L)-\area (K\cap L)$ from above.
\begin{equation}\label{eq:symdif}
\area (L)-\area (K\cap L)=\area (K\cup L)-\area (K)\leq\area (K^{(\delta(K,L))})-\area (K).
\end{equation}
The rest of the proof is divided into two cases depending on $\M^2$.
\setcounter{case}{0}
\begin{case}
$\M^2=\hyp^2$.
\end{case}
According to the hyperbolic Steiner formula~\ref{lem:steiner},
\begin{equation}
\area (K^{(\delta(K,L))})-\area (K)=(\cosh\delta(K,L)-1)(\area (K)+2\pi)+\sinh\delta(K,L)\per (K).
\end{equation}
Since $\area (K)\leq\area \left (B\big(0,\frac{D}{2}\big)\right )=2\pi\big(\cosh\frac{D}{2}-1\big)$ from the hyperbolic isodiametric inequality and $\per (K)\leq\per \left (B\big(0,\frac{D}{2}\big)\right )=2\pi\sinh\frac{D}{2}$ from Lemma~\ref{lem:maximal-perimeter}, we obtain
\begin{equation}
\area (L)-\area (K\cap L)\leq 2\pi\Big(\cosh\frac{D}{2}(\cosh\delta(K,L)-1)+\sinh\frac{D}{2}\sinh\delta(K,L)\Big)
\end{equation}
from \eqref{eq:symdif}. We can give the same bound for $\area (K)-\area (K\cap L)$, and hence
\begin{equation}
\Delta(K,L)\leq 4\pi\Big(\cosh\frac{D}{2}(\cosh\delta(K,L)-1)+\sinh\frac{D}{2}\sinh\delta(K,L)\Big).
\end{equation}
Equality could hold only if $K$ and $L$ are both circles of radius $\frac{D}{2}$, but then $\area(K\cup L)$ and $\area\left( K^{(\delta(K,L))}\right)$ can be equal only if $K=L$.
\begin{case}
$\M^2=\sph^2$.
\end{case}
From the spherical Steiner formula~\ref{lem:steiner}, we get
\begin{equation}
\area \left (K^{(\delta(K,L))}\right )-\area (K)=(1-\cos\delta(K,L))(2\pi-\area (K))+\sin\delta(K,L)\per (K).
\end{equation}

Again, we use Lemma~\ref{lem:maximal-perimeter} to bound $\per (K)$ from above. Let $T$ be a regular triangle of diameter (height) $D$ if $D\geq\frac{\pi}{2}$, and a Reuleaux triangle of diameter $D$ if $D<\frac{\pi}{2}$. Then
\begin{equation}
\per (K)\leq\per (T)=\begin{cases}
    6\sin D \arcsin\frac{1}{2\cos\frac{D}{2}}, &\text{if }D<\frac{\pi}{2},\\
    \adam{6\arccos\frac{\cos D+\sqrt{\cos^2 D+8}}{4}}, &\text{if }\frac{\pi}{2}\leq D<\pi.
\end{cases}
\end{equation}
Thus,
\begin{equation}
\area (L)-\area (K\cap L)<2\pi(1-\cos\delta(K,L))+6\sin\delta(K,L)\sin D \arcsin\frac{1}{2\cos\frac{D}{2}}
\end{equation}
if $D<\frac{\pi}{2}$, and
\begin{equation}
\area (L)-\area (K\cap L)<2\pi(1-\cos\delta(K,L))+6\sin\delta(K,L)\adam{\arccos\frac{\cos D+\sqrt{\cos^2 D+8}}{4}}
\end{equation}
if $\frac{\pi}{2}\leq D<\pi$. Using the same estimates for $\area (K)-\area (K\cap L)$, we obtain the bounds
\begin{equation}
\Delta(K,L)<\begin{cases}
    4\pi(1-\cos\delta(K,L))+12\sin\delta(K,L)\sin D \arcsin\frac{1}{2\cos\frac{D}{2}}, &\text{if }D<\frac{\pi}{2},\\
    4\pi(1-\cos\delta(K,L))+12\sin\delta(K,L)\adam{\arccos\frac{\cos D+\sqrt{\cos^2 D+8}} {4}}, &\text{if }\frac{\pi}{2}\leq D<\pi.
\end{cases}
\end{equation}
\end{proof}

Using Taylor expansions in the estimates of Lemma~\ref{lem:noneuclidean-groemer} for the trigonometric and hyperbolic functions in terms of $\delta(K,L)$, we obtain the statements of Lemma~\ref{lem:groemer-noneuclidean-taylor}.


Now we are ready to prove Theorem~\ref{thm:stab:symmetric-difference}.
\begin{proof}[Proof of Theorem~\ref{thm:stab:symmetric-difference}]
Let $K$ be an $r$-ball convex body in $\M^2$ with $\area (K)=\area (T_{w,r})+\varepsilon$ with $\varepsilon\in (0,\overline{\varepsilon})$. Note that the $r$-ball convexity of $K$ implies that $\diam K\leq 2r$. If $\M^2=\R^2$, then the theorem is a direct consequence of Theorem~\ref{thm:stab:isominwidth:spindleconvex} and \eqref{eq:groemer}. For $\M^2=\hyp^2$ and $\M^2=\sph^2$, the result follows from Theorem~\ref{thm:stab:isominwidth:spindleconvex} and Lemma~\ref{lem:groemer-noneuclidean-taylor}.
\end{proof}

\section{Optimality of Theorems~\ref{thm:stab:isominwidth:spindleconvex} and \ref{thm:stab:symmetric-difference}}\label{sec:optimality}

Finally, we show that the order of $\varepsilon$ is optimal in Theorem~\ref{thm:stab:isominwidth:spindleconvex}. Consider the $r$-disk triangle $T_{w,r}$ with vertices $v_1,v_2,v_3$, and let $v(x)$ be a point such that $v_3$ is contained in the shorter $r$-arc between $v_2$ and $v(x)$ and $d(v_3,v(x))=x$; see Figure~\ref{fig:extended-triangle}. Let $T_\varepsilon$ be the $r$-disk triangle with vertices $v_1,v_2,v(x)$. We set $x$ to be the unique value such that $\area(T_\varepsilon)=\area(T_{w,r})+\varepsilon$. We assume that $\varepsilon$ is sufficiently small so that $x<1$.

\begin{figure}[H]
\centering
\begin{tikzpicture}[line cap=round,line join=round,>=triangle 45,x=1cm,y=1cm,scale=2.5]
\clip(-1.2660254037844387,-0.71) rectangle (1.4961130857528036,1.15);
\coordinate (v1) at (0,1);
\coordinate (v2) at (-0.8660254037844387,-0.5);
\coordinate (v3) at (0.8660254037844384,-0.5);
\coordinate (vx) at (1.0961130857528036,-0.3641897618388592);
\draw [shift={(0,1.230018015984799)},line width=0.8pt]  plot[domain=4.248271477124485:5.176506483644894,variable=\t]({1*1.9346737026258407*cos(\t r)+0*1.9346737026258407*sin(\t r)},{0*1.9346737026258407*cos(\t r)+1*1.9346737026258407*sin(\t r)});
\draw [shift={(0,1.230018015984799)},line width=0.8pt]  plot[domain=5.176506483644894:5.314716891766352,variable=\t]({1*1.9346737026258412*cos(\t r)+0*1.9346737026258412*sin(\t r)},{0*1.9346737026258412*cos(\t r)+1*1.9346737026258412*sin(\t r)});
\draw [shift={(-1.0652268489553698,-0.6150090079923993)},line width=0.8pt]  plot[domain=0.059481272338094374:0.987716278858503,variable=\t]({1*1.9346737026258405*cos(\t r)+0*1.9346737026258405*sin(\t r)},{0*1.9346737026258405*cos(\t r)+1*1.9346737026258405*sin(\t r)});
\draw [shift={(1.0652268489553693,-0.6150090079923999)},line width=0.8pt]  plot[domain=2.15387637473129:3.0821113812516985,variable=\t]({1*1.934673702625841*cos(\t r)+0*1.934673702625841*sin(\t r)},{0*1.934673702625841*cos(\t r)+1*1.934673702625841*sin(\t r)});
\draw [shift={(-0.7970378284382158,-0.762865007784338)},line width=0.8pt]  plot[domain=0.2075554848346729:1.1461754836509392,variable=\t]({1*1.934673702625841*cos(\t r)+0*1.934673702625841*sin(\t r)},{0*1.934673702625841*cos(\t r)+1*1.934673702625841*sin(\t r)});
\begin{scriptsize}
\draw [fill=black] (0,1) circle (1pt);
\draw [fill=black] (-0.8660254037844387,-0.5) circle (1pt);
\draw [fill=black] (0.8660254037844384,-0.5) circle (1pt);
\draw [fill=black] (1.0961130857528036,-0.3641897618388592) circle (1pt);
\node[above] at (v1) {$v_1$};
\node[below left] at (v2) {$v_2$};
\node[below right] at (v3) {$v_3$};
\node[right] at (vx) {$v(x)$};
\end{scriptsize}
\end{tikzpicture}
\caption{The $r$-disk triangle $T_\varepsilon$}
\label{fig:extended-triangle}
\end{figure}
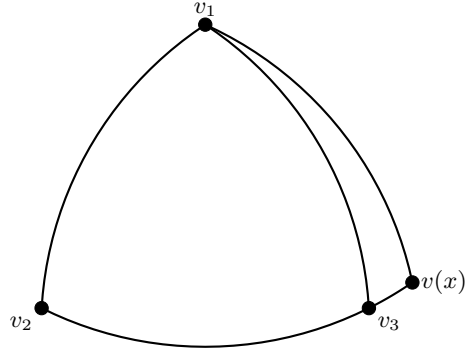

First, we notice that $T_\varepsilon$ has minimal width $w$. Indeed, the minimal width of $T_\varepsilon$ is at least $w$, since $T_\varepsilon\supset T_{w,r}$. On the other hand, $w(T_\varepsilon,\ell)=w$ where $\ell$ is the supporting line at the midpoint $m_1$ of the $r$-arc between $v_2$ and $v_3$ (see \cite{FRS26}*{Lemma 2.2}).

Clearly, Figure~\ref{fig:extended-triangle} shows the optimality of Theorem~\ref{thm:stab:symmetric-difference}.

We also observe that the diameter of $T_\varepsilon$ is
\[
\diam T_\varepsilon=d(v_2,v(x))<d(v_2,v_3)+d(v_3,v(x))=\diam T_{w,r}+x.
\]
We also bound the diameter of $T_\varepsilon$ from below.

\begin{lemma}\label{lem:diameterbound}
For a sufficiently small $\varepsilon$,
\[
\diam T_\varepsilon\geq\diam T_{w,r}+\frac{x}{2}.
\]
\end{lemma}

The following statement follows easily from the triangle inequality.

\begin{lemma}
Let $X,Y\subset\M^n$ be non-empty compact sets with $\diam X\leq\diam Y$. Then,
\[
\frac{1}{2}(\diam Y-\diam X)\leq\delta(X,Y).
\]
\end{lemma}

\begin{lemma}\label{lem:epsilon-and-x}
There is a positive constant $c_5$ such that $\varepsilon<c_5 x$.
\end{lemma}

\begin{proof}
Fix an $r$-disk $B=B(o,r)\subset\M^2$ and a point $z\in\partial B$. Denote the circular sector of $B$ centered at $o$ and chord $[z,z(t)]$ of length $t\in [0,2r]$ by $S(t)$ where $o,z,z(t)$ have a positive orientation, and denote the circular segment defined by $S(t)$ by $S_0(t)$. We observe that $d(v_1,v_3)<d(v_1,v(x))\leq r$. We call the Lexell curve $\mathcal{L}$ of the triangle $[v_1,v_3,v(x)]$ with respect to the side $[v_3,v(x)]$ the locus of points $z$ with $z$ on the same half-plane bounded by the line through $v_3$ and $v(x)$ such that $\area([v_1,v_3,v(x)])=\area([z,v_3,v(x)])$. It is known that $\mathcal{L}$ is the line through $v_1$ parallel to $[v_3,v(x)]$ if $\M^2=\R^2$, it is a hypercycle (an equidistant curve to a line) with the same ideal points as $[v_3,v(x)]$ for $\M^2=\hyp^2$, and it is a circular arc if $\M^2=\sph^2$. As a consequence, we can always find a point $z_0$ on $\mathcal{L}$ such that $d(z_0,v_3)=d(z_0,v(x))$, and in this case the lengths of the legs of the isosceles triangle $[z_0,v_3,v(x)]$ are between $d(v_1,v_3)$ and $d(v_1,v(x))$. Hence, the area of the region bounded by the $r$ arc between $v_3$ and $v(x)$ and by the segments $[v(x),v_1]$ and $[v_1,v_3]$ can be bounded from above by the area of a circular sector $S(x)$. The area $\varepsilon$ of the region bounded by the $r$-arcs between $v_1$ and $v_3$, between $v_3$ and $v(x)$ and between $v(x)$ and $v_1$ can be expressed as the area of the region bounded by the $r$ arc between $v_3$ and $v(x)$ and by the segments $[v(x),v_1]$ and $[v_1,v_3]$ increased by the area difference between $S_0(d(v_1,v(x))$ and $S_0(d(v_1,v_3)$, thus
\[
\varepsilon<\area(S(x))+\area(S_0(d(v_1,v(x)))-\area(S_0(d(v_1,v_3)).
\]
Let $\gamma(t)$ be the central angle of $S(t)$. From the law of cosines,
\[
\cos\gamma(t)=\begin{cases}
    1-\frac{t^2}{2r^2},&\text{if }\M^2=\R^2,\\
    \frac{\cosh^2 r-\cosh^2 t}{\sinh^2 t},&\text{if }\M^2=\hyp^2,\\
    \frac{\cos^2 t-\cos^2 r}{\sin^2 t},&\text{if }\M^2=\sph^2.
\end{cases}
\]
Hence,
\begin{equation}\label{eq:area-of-sector}
\area(S(x))=\begin{cases}
    \frac{r^2\gamma(x)}{2}=\frac{r}{2}x+O(x^3),&\text{if }\M^2=\R^2,\\
    (\cosh r-1)\gamma(x)=\frac{\cosh r-1}{\sinh r}x+O(x^3),&\text{if }\M^2=\hyp^2,\\
    (1-\cos r)\gamma(x)=\frac{1-\cos r}{\sin r}x+O(x^3),&\text{if }\M^2=\sph^2.
\end{cases}
\end{equation}
We notice that $\area(S_0(t))$ increases in $t$ and that $d(v_1,v(x))<d(v_1,v_3)+x$ from the triangle inequality, hence $\area(S_0(d(v_1,v(x)))-\area(S_0(d(v_1,v_3))\leq\area(S_0(d(v_1,v_3)+x)-\area(S_0(d(v_1,v_3))$. 
Therefore,
\begin{equation}\label{eq:difference-between-circular-segments}
\area(S_0(d(v_1,v_3)+x)-\area(S_0(d(v_1,v_3))=\begin{cases}
    \frac{d(v_1,v_3)^2}{2\sqrt{4r^2-d(v_1,v_3)^2}}x+O(x^2),&\text{if }\M^2=\R^2,\\
    \frac{\cosh r\sinh d(v_1,v_3)\tanh^2\frac{d(v_1,v_3)}{2}}{\sqrt{\sinh^4 r-(\cosh^2 r-\cosh d(v_1,v_3))}}x+O(x^2),&\text{if }\M^2=\hyp^2,\\
    \frac{2\sqrt{2}\cos r\sin^3\frac{d(v_1,v_3)}{2}}{\sqrt{\cos d(v_1,v_3)-\cos(2r)}}x+O(x^2),&\text{if }\M^2=\sph^2
\end{cases}
\end{equation}
from Lemma~\ref{lem:area-of-circular-segment}. The statement of the lemma follows from \eqref{eq:area-of-sector} and \eqref{eq:difference-between-circular-segments}.
\end{proof}

\begin{corollary}
Theorem~\ref{thm:stab:isominwidth:spindleconvex} is optimal.
\end{corollary}
\begin{proof}
Consider the $r$-disk triangle $T_\varepsilon$ defined as above. Then, for an arbitrary $\phi\colon\M^2\to\M^2$ isometry,
\[
\delta(T\varepsilon,\phi T_{w,r})\geq\frac{1}{2}(\diam T_\varepsilon-\diam T_{w,r})>\frac{x}{4}>\frac{1}{4c_5}\varepsilon
\]
from Lemma~\ref{lem:diameterbound} and Lemma~\ref{lem:epsilon-and-x}.
\end{proof}

\appendix
\section{}

\feri{
Finally, we prove Lemma~\ref{lem:monotonicity}. We start with two technical statements. We omit the proofs as they are straightforward.

\begin{prop}\label{prop:f-incresing-hyp}
For $1<p\leq q$, let
\[
f_\hyp(p,q)=\frac{4q-pq+\sqrt{(p^2-1)(q^2-1)}}{6}q-\frac{\sqrt{q^2-1}+1}{2}
\sqrt{\frac{\left (4q-pq+\sqrt{(p^2-1)(q^2-1)}\right)^2}{9}-1}.
\]
Then $\frac{\partial f_\hyp}{\partial p}(p,q)>0$.
\end{prop}

\begin{prop}\label{prop:f-incresing-sph}
For $1>p\geq q>0$, let
\[
f_\sph(p,q)=\frac{4q-pq-\sqrt{(1-p^2)(1-q^2)}}{6}q+\frac{\sqrt{1-q^2}+1}{2}\sqrt{1-\frac{\left (4q-pq-\sqrt{(1-p^2)(1-q^2)}\right )^2}{9}}.
\]
Then $\frac{\partial f_\sph}{\partial p}(p,q)>0$.
\end{prop}

\begin{proof}[Proof of Lemma~\ref{lem:monotonicity}]
We note that due to symmetry, it is enough to see that the area of the region bounded by the geodesic segments $\left[p,q_1\right]$, $\left[p,t_3\right]$ and the arc between $q_1$ and $t_3$ on the boundary of $Q_{w,r,\varrho}$ increases in $\varrho$. We denote this region as $P_\varrho$. Let $c$ be the center of the circle that contains this arc on its boundary. We observe that
\begin{equation}\label{eq:area:t3pq1}
\area\left(P_\varrho\right)=\area\left(\left[t_3,p,q_1\right]\right)+\area\left(S\right)-\area\left(\left[t_3,c,q_1\right]\right),
\end{equation}
where $S$ is the sector of the circle $B\left(c,r\right)$ between $q_1$ and $t_3$. We observe that $d\left(t_3,p\right)=\varrho$, $d\left(q_1,p\right)=w-\varrho$ and $\angle\left(t_3,p,q_1\right)=\frac{\pi}{3}$. We also use the notations $s=d\left(q_1,t_3\right)$, $\alpha=\angle\left(q_1,q_2,q_3\right)$ and $\beta=\angle\left(t_3,c,q_1\right)$. The rest of the argument is divided into three cases depending on the plane $\M^2$.

\setcounter{case}{0}
\begin{case}
    $\M^2=\R^2$.
\end{case}

Clearly, $\area\left(\left[t_3,p,q_1\right]\right)=\frac{\varrho\left(w-\varrho\right)\sqrt{3}}{4}$. By the law of cosines, we also have
\begin{equation}\label{eq:lawofcos:equcl}
s^2=\varrho^2+\left(w-\varrho\right)^2-\varrho\left(w-\varrho\right)=2r^2-2r^2\cos\beta,
\end{equation}
hence
\begin{equation}\label{eq:beta:eukl}
\beta=\arccos\left(1-\frac{s^2}{2r^2}\right).
\end{equation}

Since $\area\left(S\right)=\frac{\beta}{2}r^2$ and $\area\left(\left[t_3,c,q_1\right]\right)=\frac{r^2\sin\beta}{2}$, we have
$$
\area\left(P_\varrho\right)=\frac{\varrho\left(w-\varrho\right)\sqrt{3}}{4}+\frac{r^2}{2}\left(\beta-\sin\beta\right).
$$
From the first equation of \eqref{eq:lawofcos:equcl} it follows that
one can interpret the area as a function of $s$; let
$$
a\left(s\right)=\frac{\sqrt{3}}{12}\left(w^2-s^2\right)+\frac{r^2}{2}\left(\arccos\left(1-\frac{s^2}{2r^2}\right)-\sqrt{\frac{s^2}{r^2}-\frac{s^4}{4r^4}}\right).
$$
It is straightforward to check that the derivative of $a$ is negative for $0<s<r$, and that $s$ decreases as $\varrho$ increases. Hence, $a\left(s\right)=\area\left(P_\varrho\right)$ increases in $\varrho$.

\begin{case}
    $\M^2=\hyp^2$.
\end{case}

To determine $\beta$, we use the hyperbolic law of cosines twice, similarly to the Euclidean case. However, 
it is more convenient to express the area in terms of $t=\cosh\frac{w-2\varrho}{2}$ instead of $s$.
Then
\begin{equation}\label{eq:area:t3pq1:hyp}
\area\left(\left[t_3,p,q_1\right]\right)=2\arccot\left(\frac{2}{\sqrt{3}}\left (\frac{\cosh\frac{w}{2}+t}{\cosh\frac{w}{2}-t}-\frac{1}{2}\right )\right)
\end{equation}
by the hyperbolic area formula in Lemma~\ref{lemma:triangle-area}.

By the hyperbolic law of cosines applied to the triangles $\left[t_3,p,q_1\right]$ and $\left[t_3,c,q_1\right]$, we have the equations
\begin{equation}\label{eq:lawofcos:hyp}
\cosh s=\cosh\varrho\cosh\left(w-\varrho\right)-\frac{1}{2}\sinh\varrho\sinh\left(w-\varrho\right)=\cosh^2 r-\sinh^2 r\cos\beta,
\end{equation}
and applying hyperbolic trigonometric identities to the first equation of \eqref{eq:lawofcos:hyp}, we obtain
\begin{equation}\label{eq:coshs}
\begin{split}
\cosh s
=-1+\frac{1}{2}\cosh^2 \frac{w}{2}+\frac{3}{2}t^2.
\end{split}
\end{equation}
From \eqref{eq:lawofcos:hyp} and \eqref{eq:coshs},
\begin{equation}\label{eq:beta:hyp}
\beta=\arccos\frac{\cosh^2 r+1-\frac{1}{2}\cosh^2\frac{w}{2}-\frac{3}{2}t^2}{\sinh^2 r}.
\end{equation}
Using \eqref{eq:beta:hyp}, the area of the sector of angle $\beta$ is
\begin{equation}\label{eq:area:S:hyp}
\area\left(S\right)=\left(\cosh r-1\right)\beta=\left(\cosh r-1\right)\arccos\frac{\cosh^2 r+1-\frac{1}{2}\cosh^2\frac{w}{2}-\frac{3}{2}t^2}{\sinh^2 r},
\end{equation}
while from Lemma~\ref{lemma:triangle-area} and \eqref{eq:beta:hyp}, the area of the isosceles triangle $\left[t_3,c,q_1\right]$ is
\begin{equation}\label{eq:area:t3cq1:hyp}
\area\left(\left[t_3,c,q_1\right]\right)=2\arccot\left(\frac{\coth^2\frac{r}{2}-\frac{\cosh^2 r+1-\frac{1}{2}\cosh^2\frac{w}{2}-\frac{3}{2}t^2}{\sinh^2 r}}{\sqrt{1-\left(\frac{\cosh^2 r+1-\frac{1}{2}\cosh^2\frac{w}{2}-\frac{3}{2}t^2}{\sinh^2 r}\right)^2}}\right).
\end{equation}
Combining \eqref{eq:area:t3pq1:hyp}, \eqref{eq:area:S:hyp} and \eqref{eq:area:t3cq1:hyp}, we get the area $a\left(t\right)$ of $P_\varrho$, whose
derivative is
\begin{equation}
a'\left(t\right)=-\frac{4\sqrt{3}\cosh\frac{w}{2}}{1+6t^2+\cosh w}+\frac{12t\cosh r\sqrt{-7+6t^2+\cosh w}}{\left(1+6t^2+\cosh w\right)\sqrt{\left(3-6t^2+4\cosh\left(2r\right)-\cosh w\right)}}.
\end{equation}
As $\varrho\in\left[\varrho_0,\frac{w}{2}\right)$ increases, $t$ decreases, so it is enough to show that $a'\left(t\right)$ is negative, or equivalently
\begin{equation}
\begin{split}
\left(18\cosh^2 r\right)t^4+\left(6\cosh^2\frac{w}{2}-21\cosh^2 r+3\cosh^2 r\cosh w\right)t^2+\\
\left(\cosh w\cosh^2\frac{w}{2}-3\cosh^2\frac{w}{2}-4\cosh^2\frac{w}{2}\cosh(2r)\right)<0.
\end{split}
\end{equation}
This is a quadratic inequality in the form of $Q(t^2)<0$, so it holds when $t^2\in\left(X_1,X_2\right)$ where $X_1<X_2$ are the roots of $Q$. It can be easily checked that $Q(0)<0$ and that the leading coefficient of $Q$ is positive, so there are indeed two different real roots $X_1<X_2$. Recall that $t=\cosh\frac{w-2\varrho}{2}$ decreases in $\varrho$, and since $X_1<0$, it is enough to show
\begin{equation}
\cosh^2\frac{w-2\varrho_0}{2}<X_2.
\end{equation}
Now let $p=\cosh w$ and $q=\cosh r$. We can express $X_2$ in terms of $p$ and $q$:
\begin{equation}
X_2=\frac{4}{3}-\frac{p+1}{6q^2}=:g_\hyp(p,q).
\end{equation}
We can also express $\cosh^2\frac{w-2\varrho_0}{2}$ in terms of $p$ and $q$. We substitute $\varrho_0$ from Lemma \ref{lem:rho0}:
\begin{equation}
\cosh^2\frac{w-2\varrho_0}{2}=f_\hyp(p,q).
\end{equation}
In other words, we want to verify
\begin{equation}
f_\hyp(p,q)<g_\hyp(p,q)
\end{equation}
whenever $1<p\leq q$. We observe that $\frac{\partial f_\hyp}{\partial p}(p,q)$ is positive in this domain (see Proposition~\ref{prop:f-incresing-hyp}) and clearly $\frac{\partial g_\hyp}{\partial p}(p,q)$ is negative, so $f_\hyp(p,q)\leq f_\hyp(q,q)$ and $g_\hyp(q,q)\leq g_\hyp(p,q)$. Hence, it is sufficient to prove
\begin{equation}
f_\hyp(q,q)<g_\hyp(q,q),
\end{equation}
but that is easy to see.

\begin{case}
$\M^2=\sph^2$
\end{case}

Now we set $t=\cos\frac{w-2\varrho}{2}$. Using basic trigonometric identities, we obtain
$$
\cot\frac{\varrho}{2}\cot\frac{w-\varrho}{2}=\frac{t+\cos\frac{w}{2}}{t-\cos\frac{w}{2}}.
$$
Thus, the area of the triangle $\left[t_3,p,q_1\right]$ can be derived from the area formula in Lemma~\ref{lemma:triangle-area}:
\begin{equation}\label{eq:area:t3pq1:sph}
\area\left(\left[t_3,p,q_1\right]\right)=2\arccot\left(\frac{2}{\sqrt{3}}\left (\frac{t+\cos\frac{w}{2}}{t-\cos\frac{w}{2}}+\frac{1}{2}\right )\right).
\end{equation}
Now we apply the spherical law of cosines to the triangles $\left[t_3,p,q_1\right]$ and $\left[t_3,c,q_1\right]$ to get
\begin{equation}\label{eq:lawofcos:sph}
\cos s=\cos\varrho\cos\left(w-\varrho\right)+\frac{1}{2}\sin\varrho\sin\left(w-\varrho\right)=\cos^2 r+\sin^2 r\cos\beta.
\end{equation}
From the first equation of \eqref{eq:lawofcos:sph}, we get
\begin{equation}\label{eq:coss}
\cos s=-1+\frac{1}{2}\cos^2\frac{w}{2}+\frac{3}{2}t^2.
\end{equation}
From \eqref{eq:lawofcos:sph} and \eqref{eq:coss},
\begin{equation}\label{eq:beta:sph}
\beta=\arccos\frac{-\cos^2 r-1+\frac{1}{2}\cos^2\frac{w}{2}+\frac{3}{2}t^2}{\sin^2 r}.
\end{equation}
From \eqref{eq:beta:sph}, the area of the sector of angle $\beta$ is
\begin{equation}\label{eq:area:S:sph}
\area\left(S\right)=\left(1-\cos r\right)\beta=\left(1-\cos r\right)\arccos\frac{-\cos^2 r-1+\frac{1}{2}\cos^2\frac{w}{2}+\frac{3}{2}t^2}{\sin^2 r}.
\end{equation}
Finally, the are of the triangle $\left[t_3,c,q_1\right]$ is
\begin{equation}\label{eq:area:t3cq1:sph}
\area\left(\left[t_3,c,q_1\right]\right)=2\arccot\left(\frac{\cot^2\frac{r}{2}+\frac{-\cos^2 r-1+\frac{1}{2}\cos^2\frac{w}{2}+\frac{3}{2}t^2}{\sin^2 r}}{\sqrt{1-\left(\frac{-\cos^2 r-1+\frac{1}{2}\cos^2\frac{w}{2}+\frac{3}{2}t^2}{\sin^2 r}\right)^2}}\right)
\end{equation}
from Lemma~\ref{lemma:triangle-area} and \eqref{eq:beta:sph}. Combining \eqref{eq:area:t3pq1:sph}, \eqref{eq:area:S:sph} and \eqref{eq:area:t3cq1:sph}, we get the area $a\left(t\right)$ of $P_\varrho$, whose derivative is
\begin{equation}
a'\left(t\right)=\frac{4\sqrt{3}\cos\frac{w}{2}}{1+6t^2+\cos w}-\frac{12t\cos r\sqrt{7-6t^2-\cos w}}{\left(1+6t^2+\cos w\right)\sqrt{\left(-3+6t^2-4\cos\left(2r\right)+\cos w\right)}}.
\end{equation}
As $\varrho\in\left[\varrho_0,\frac{w}{2}\right)$ increases, $t$ increases, so it is enough to show that $a'\left(t\right)$ is positive, or equivalently $Q(t^2)>0$,
with
\begin{equation}
\begin{split}
Q(X)=\left(18\cos^2 r\right)X^2+\left(6\cos^2\frac{w}{2}-21\cos^2 r+3\cos^2 r\cos w\right)X+\\
\left(\cos w\cos^2\frac{w}{2}-3\cos^2\frac{w}{2}-4\cos^2\frac{w}{2}\cos(2r)\right).
\end{split}
\end{equation}
We note that $Q$ has two different real roots $X_1<X_2$, since the leading coefficient is positive and $Q(0)<0$. This also means that $X_1<0$, so we only need to show that $t^2>X_2$ if $\varrho\in\left[\varrho_0,\frac{w}{2}\right)$. The minimal value of $t^2$ on this interval is $\cos^2\frac{w-2\varrho_0}{2}$, so we want to prove the inequality
\begin{equation}
X_2<\cos^2\frac{w-2\varrho_0}{2}.
\end{equation}
For the rest of the proof, let $p=\cos w$ and $q=\cos r$. Then we observe that $\cos^2\frac{w-2\varrho_0}{2}=f_\sph(p,q)$ and identically to the hyperbolic case,
\begin{equation}
g_\sph(p,q):=X_2=\frac{4}{3}-\frac{p+1}{6q^2}.
\end{equation}
We saw that $\frac{\partial f_\sph}{\partial p}(p,q)$ is positive in the domain $1>p\geq q>0$ (see Proposition~\ref{prop:f-incresing-sph}) and that $\frac{\partial g_\sph}{\partial p}(p,q)$ is negative, so as $p\geq q$, $f_\sph(p,q)\geq f_\sph(q,q)$ and $g_\sph(q,q)\geq g_\sph(p,q)$. Thus, we only need to show that
\begin{equation}
f_\sph(q,q)>g_\sph(q,q),
\end{equation}
which is again straightforward.
\end{proof}
}

\section*{Acknowledgments}
This research was supported by NKFIH projects no.~150151 and no.~150613. Projects no.~150151 and no.~150613 have been implemented with the support provided by the Ministry of Culture and Innovation of Hungary from the National Research, Development and Innovation Fund, financed under the ADVANCED\_24 funding scheme.

The second author received funding from project no. 2024-1.2.8-T\'ET-IPARI-CN-2025-00011,
with the support provided by the Ministry of Innovation and Technology
of Hungary from the National Research, Development and Innovation Fund.

This research was also supported by project TKP2021-NVA-09. Project no. TKP2021-NVA-09 has been implemented with the support provided by the Ministry of Innovation and Technology of Hungary from the National Research, Development and Innovation Fund, financed under the TKP2021-NVA funding scheme.

\end{document}